\documentclass[11pt,a4paper,twoside]{article}       
\usepackage[utf8]{inputenc}
\usepackage[english]{babel}
\usepackage{amssymb, amsmath, amsfonts, amsthm}
\usepackage{enumerate}
\usepackage{colonequals}
\usepackage{empheq,fancybox}
\usepackage[small]{titlesec}
\usepackage[noblocks]{authblk}

\usepackage{microtype}
\usepackage{ellipsis}
\usepackage[BCOR=20mm,DIV=11]{typearea}
 \newcommand{\hm}[1]{\leavevmode{\marginpar{\tiny%
 $ \hbox to 0mm{\hspace*{-0.5mm} $ \leftarrow $ \hss}%
 \vcenter{\vrule depth 0.1mm height 0.1mm width \the\marginparwidth}%
 \hbox to
 0mm{\hss $ \rightarrow $ \hspace*{-0.5mm}} $ \\\relax\raggedright #1}}}

\newcommand{\euler}{\mathrm{e}} 
\newcommand{\drm}{\mathrm{d}}
\newcommand{\dvol}{\mathrm{dvol}}
\newcommand{\RR}{\mathbb{R}}
\newcommand{\CC}{\mathbb{C}}
\newcommand{\NN}{\mathbb{N}}	
\newcommand{\ZZ}{\mathbb{Z}}

\newcommand{\HS}{\mathrm{HS}}

\newcommand{\tvert}[1]{{\left\vert\kern-0.25ex\left\vert\kern-0.25ex\left\vert #1 
    \right\vert\kern-0.25ex\right\vert\kern-0.25ex\right\vert}}
\renewcommand{\epsilon}{\varepsilon}

\DeclareMathOperator{\diam}{\mathop{diam}}

\DeclareMathOperator{\Ric}{\mathop{Ric}}
\DeclareMathOperator{\Vol}{\mathop{Vol}}

\DeclareMathOperator{\Ker}{\mathop{Ker}}

\DeclareMathOperator{\esssup}{esssup}
\newtheorem{theorem}{Theorem}[section]

\newtheorem{proposition}[theorem]{Proposition}
\newtheorem{corollary}[theorem]{Corollary}
\theoremstyle{definition}

\theoremstyle{remark}
\newtheorem{remark}[theorem]{Remark}
\theoremstyle{example}
\newtheorem{example}[theorem]{Example}
\newcommand{\hil}{\mathcal H}
\newcommand{\KK}{\mathbb{K}}	

\usepackage{color}
	\definecolor{darkred}{rgb}{0.5,0,0}
	\definecolor{darkgreen}{rgb}{0,0.5,0}
	\definecolor{darkblue}{rgb}{0,0,0.5}
\usepackage{authblk}

\begin{document}
\title{Bounds on the first Betti number - an approach via Schatten norm estimates on semigroup differences} 
%

%
%
\author[1]{Marcel Hansmann}
\author[2]{Christian Rose}
\author[1]{Peter Stollmann}
\affil[1]{Technische Universit\"at Chemnitz, Faculty of Mathematics, D - 09107 Chemnitz}
\affil[2]{Max-Planck Institute for Mathematics in the Sciences, D - 04103 Leipzig}

\date{}

\maketitle


\begin{abstract}
We derive new estimates for the first Betti number of compact Riemannian manifolds. Our approach relies on the Birman-Schwinger principle and Schatten norm estimates for semigroup differences. In contrast to previous works we do not require any a priori ultracontractivity estimates and we provide  bounds which explicitly depend on suitable integral norms of the Ricci tensor. 
\end{abstract}

\section{Introduction}   
The aim of this paper is to give estimates for the first  Betti number $b_1(M)$ of a compact  Riemannian manifold $M$. In particular, we show that $b_1(M)$ is small, or even zero, if the Ricci curvature of $M$ is 'mostly positive'. 

A starting point for our analysis is the paper \cite{ElworthyRosenberg-91}, where
a criterion is formulated under which $b_1(M)=0$. The results of \cite{ElworthyRosenberg-91} were later generalized and put into a more quantitative form in \cite{RoseStollmann-18} and \cite{CarronRose-18}, respectively (for some related work see also the literature cited in these two papers). The new feature of our main results can easily be explained: Roughly speaking, we employ methods 
from Functional Analysis and Operator Theory that originated in Mathematical Physics. This complements the standard approach using the Hodge theorem to identify the first Betti number with the dimension of the kernel of the  Laplacian acting on $1$--forms, and then to deduce bounds on $b_1(M)$ from trace norm estimates of the corresponding heat semigroup. However, the method we use in the present paper is different: In case that the Ricci curvature is mostly positive, an alternative, and obviously better, estimate can be obtained by means of the Birman--Schwinger principle, see Section \ref{sec2} and Proposition \ref{prop:birman}. In particular, instead of estimating
the trace norm of the semigroup itself, it suffices to bound the trace norm (or a more general Schatten norm) of an appropriate semigroup difference.

We derive such trace norm bounds by well--established factorization principles; however
the abstract results we obtain in Proposition \ref{prop:2.6} and Theorems \ref{thm:truncated} and \ref{thm:ultra} are new and of independent interest. The application of these abstract trace norm bounds to the geometric setting is then straightforward, as we show in Section~\ref{sec3}. Our main result, Theorem \ref{main}, gives a very satisfactory estimate:  For any compact Riemannian manifold $M$ of dimension $n$ and for $\rho_0>0$ and $t_0>0$, we have
 \begin{equation}\label{bettimain2}
 b_1(M)\le 4n\rho_0^{-2}\left\| (\Ric_\cdot-\rho_0)_-\right\|_{2,\HS}^2\left\| e^{-t_0(\Delta+\rho)}\right\|^2_{2,\infty} .
   \end{equation}
A detailed explanation of the notation used in this inequality will be given below. Let us just note that here the $L^2$--norm of the part of the Ricci tensor $\Ric$ below the positive threshold $\rho_0$ is used to control that $\Ric$ is indeed mostly positive. The ultracontractivity term on the right still depends on geometric data via the function $\rho$, which maps every point of the manifold to the lowest eigenvalue of the Ricci tensor. In Corollary \ref{cor:main}, we will exemplify how this part of the estimate can be controlled in geometric terms. 

As far as we can say, the above bound on $b_1(M)$ is the first containing a norm of the matrix-valued map $M\ni x\mapsto(\Ric_x-\rho_0)_-$ in an explicit and multiplicative way. All previous works on this matter were formulated given assumptions on the function $\rho$ only and offered right-hand sides of a much more complicated nature. Some more on this topic will be discussed in Section~\ref{sec3}.

\section{The Birman-Schwinger principle and Hilbert-Schmidt norm estimates}
\label{sec2}
%
%
%
%
Let us consider two selfadjoint  operators $H,H'$ on a Hilbert space $\hil$, 
such that $H \geq 0$ and $H' \geq \rho_0$ for some $\rho_0>0$. Then both 
operators generate strongly continuous (even analytic) semigroups, denoted by 
$(e^{-tH}; t \geq 0)$ and $(e^{-tH'}; t \geq 0)$, respectively, and in the 
following we assume that for some $t_0>0$ the semigroup difference
$$ D_{t_0}:=e^{-t_0H} - e^{-t_0H'}$$
is compact. We can think of $H$ and $H'$ as being 'small' perturbations of each 
other, the smallness being reflected in the compactness of $D_{t_0}$.
\begin{remark}
As an example the reader should have in mind the case where $\hil$ is a 
(vector-valued) $L^2$-space and $H'=H+V$ is a potential perturbation. Here the 
compactness of $D_{t_0}$  will follow from a suitable smallness assumption on 
$V$. This case will be discussed in more detail in the second part of this 
section.  
\end{remark}
From the spectral (mapping) theorem it follows that the spectra of 
generator and semigroup are related by
$$ \sigma(e^{-tH})\setminus \{0\} = \{ e^{-t \lambda} : \lambda \in 
\sigma(H)\},$$
and that the same identity is valid for the essential spectra $\sigma_{ess}$ and 
discrete spectra $\sigma_d$ as well. Moreover, for every $\lambda \in \RR$ we 
have 
$$ \ker (H-\lambda) = \ker (e^{-tH}-e^{-t\lambda}),$$
so that also the multiplicities of the corresponding eigenvalues $\lambda_0 \in 
\sigma_d(H)$ and $e^{-t\lambda_0} \in \sigma_d(e^{-tH})$ coincide. 

Together with  Weyl's classical theorem on the invariance of the essential 
spectrum under compact perturbations \cite{zbMATH02637394}, the above facts and 
assumptions imply that
$$ \sigma_{ess}(H)=\sigma_{ess}(H') \subset [\rho_0,\infty)$$
and that the spectrum of $H$ in $[0,\rho_0)$ is purely discrete. In particular, 
this shows that the dimension of the kernel of $H$ is always finite (since 
either $0$ is in the resolvent set or it is an eigenvalue of finite 
multiplicity). 

\begin{remark}
 For our main application we will consider the case where $H=\Delta^1$ denotes the Hodge-Laplacian 
on the Riemannian manifold $M$, in which case $\dim \ker(H)$ coincides with the first Betti number $b_1(M)$.
\end{remark}

Our next goal is to obtain an upper bound on $\dim \ker(H)$ given some more 
restrictive assumptions on $D_{t}$. To this end, we need to introduce the 
Schatten-von Neumann classes $\mathcal S_p, p > 0,$ which consist of all compact 
operators $K$ on $\hil$ whose sequence of singular numbers $(s_n(K))$ is in 
$l^p(\NN)$. One can define a (quasi-) norm on $\mathcal S_p$ by setting 
$\|K\|_{\mathcal S_p}:= \|(s_n(K))\|_{p}$. It is well known that $(\mathcal S_p, 
\|.\|_{\mathcal S_p})$ is a (quasi-) Banach space and a two sided ideal in the 
algebra $\mathcal L (\hil)$ of all bounded operators on $\hil$, see, e.g., 
\cite{b_Gohberg69}
\begin{remark}
Operators in $\mathcal S_1$ and $\mathcal S_2$ are usually called \emph{trace 
class} and \emph{Hilbert-Schmidt} operators, respectively. We will follow this 
practice in the present article and will also use the more suggestive notation 
$$ \|K\|_{\text{tr}}:= \|K\|_{\mathcal S_1} \quad \text{and} \quad 
\|K\|_{\text{HS}}:=\|K\|_{\mathcal S_2}.$$
For later purposes we recall that the Hilbert-Schmidt norm can be computed as 
$$\|K\|_{\text{HS}}^2 = \sum_{\alpha \in A} \|K\varphi_\alpha\|^2,$$
 where $(\varphi_\alpha)_{\alpha \in A}$ is any orthonormal basis of $\hil$. 
 
\end{remark}
In order to obtain bounds on $\dim \ker (H)$ we will rely on the so-called 
'Birman-Schwinger principle'. This principle was first stated in the context of 
Schr\"odinger operators in \cite{MR0142896, MR0129798} and originally referred 
to the following fact: Under suitable assumptions on the potential $V$ a 
negative number $\lambda$ is a discrete eigenvalue of $-\Delta+ V$ if and only 
if $1$ is in the spectrum of the compact integral operator 
$(\lambda+\Delta)^{-1}V$. The first part of the following proposition can be 
regarded as an abstract version of this principle.
\begin{proposition} \label{prop:birman}
Let $H \geq 0, H' \geq \rho_0 > 0$ and $D_{t}:=e^{-tH}-e^{-tH'},t >0,$ be 
defined as above. Then the following holds:  
\begin{enumerate}
\item[(i)] $ \ker (H) = \ker( (I-e^{-tH'})^{-1}D_{t} - I),$ where $I \in 
\mathcal L(\hil)$ denotes the identity.
\item[(ii)] If $D_{t_0} \in \mathcal S_p$ for some $p>0$ and $t_0>0$, then
  \begin{equation}
\dim \ker (H) \leq \|(I-e^{-t_0H'})^{-1}D_{t_0}\|_{\mathcal S_p}^p.\label{eq:1}
\end{equation}
In particular,
\begin{equation}
  \label{eq:2}
\dim \ker (H) \leq (1-e^{-\rho_0t_0})^{-p}\|D_{t_0}\|_{\mathcal S_p}^p.  
\end{equation}
\end{enumerate}
\end{proposition}
We note that $(I-e^{-tH'})$ is indeed invertible, since $\sigma(e^{-tH'}) 
\subset [0,e^{-\rho_0t}]$. 
In a different setting, the idea to use Schatten norm bounds on semigroup differences to obtain bounds on eigenvalues has also been used in \cite[Remark 2.6]{MR2836430} (see also \cite{MR2772053}).     
\begin{proof}[Proof of Proposition \ref{prop:birman}]
(i) We have 
$$ \ker(H) = \ker(e^{-tH}-I) = \ker((I-e^{-tH'})^{-1}D_t-I).$$
Here the first equality follows (as stated above) from the spectral theorem 
(functional calculus) and the second can be proved as follows: 
\begin{eqnarray*}
 e^{-tH}f=f \Leftrightarrow \quad D_tf = (I-e^{-tH'})f \quad \Leftrightarrow 
\quad (I-e^{-tH'})^{-1}D_tf = f. 
\end{eqnarray*}
(ii) Let $\lambda_1(K), \lambda_2(K), \ldots$ denote the sequence of eigenvalues 
of the compact operator $K:=(I-e^{-t_0H'})^{-1}D_{t_0}$, each eigenvalue being 
counted according to its algebraic multiplicity. By part (i) we know that at 
least $N:=\dim \ker(H)$ of these eigenvalues are equal to $1$, so clearly 
$$ N \leq \sum_{n} |\lambda_n(K)|^p.$$
Another classical result of Weyl \cite{MR0030693}  asserts that
$$ \sum_{n} |\lambda_n(K)|^p \leq \sum_{n} s_n(K)^p = \|K\|_{\mathcal S_p}^p,$$
which concludes the proof of (\ref{eq:1}). Finally, (\ref{eq:2}) follows from 
(\ref{eq:1}) using the estimate
$$ \|(I-e^{-t_0H'})^{-1}D_{t_0}\|_{\mathcal S_p} \leq \|(I-e^{-t_0H'})^{-1}\| 
\|D_{t_0}\|_{\mathcal S_p} \leq (1-e^{-\rho_0t_0})^{-1} \|D_{t_0}\|_{\mathcal 
S_p}.$$
\end{proof}

Our ultimate goal is to apply the previous proposition to estimate the first 
Betti number of Riemannian manifolds. To this end, as a first step we now 
provide suitable $\mathcal S_p$-norm estimates for semigroup differences on 
$L^2$-spaces. Actually, in the present paper we will restrict ourselves to the 
case of Hilbert-Schmidt norm estimates, since, in addition to being easier to 
handle, these are particularly well-suited for the applications we have in 
mind.\\

From now on we fix a $\sigma$-finite measure space $(X,\mathcal F,m)$ and 
consider the Hilbert space $L^2(X)=L^2(X;\KK)$ for $\KK \in \{ \RR, 
\CC\},$ as well as the vector-valued version $L^2(X;\KK^n)$ for $n \in \NN$. 
\begin{remark} 
  Note that for most of what we show in the following, we could replace $\KK^n$ 
by a $\KK$-Hilbert space $\mathcal K$.
\end{remark}
As usual, inner products and norms in the vector- and scalar-valued $L^2$-spaces will be denoted by the same symbols  $(\cdot\mid\cdot )$ and $\|.\|_2$, respectively. In particular, for $f=(f_1,\ldots,f_n) \in L^2(X;\KK^n)$ we have 
$ \|f\|_2^2 = \sum_{k=1}^n \|f_k\|_2^2.$  Also, we use the same symbol $|.|$ to denote the absolute value on $\KK$ as well as the euclidian norm on $\KK^n$.\\

As above we study a nonnegative 
selfadjoint operator $H$ in $L^2(X;\KK^n)$ and its semigroup $(e^{-tH}; t \geq 
0)$. In the present setting the semigroup differences we are concerned with are 
given by
$$ e^{-tH}-e^{-t(H+V)},$$
where $V$ is a multiplication operator in the following sense:
$$ V : X \to \KK^{n \times n} \text{ is measurable}$$
and for every $x \in X$ the matrix $V(x)$ is nonnegative and hermitian 
(respectively symmetric). 
The entries of this matrix are written as $V_{ij}(x)$. In the following we use 
the shorthand notation
\begin{equation}
V \succeq 0,\label{eq:9}
\end{equation}
to indicate that $V$ satisfies the above conditions. We note that 
$$ V \succeq 0 \quad \Rightarrow \quad V \geq 0,$$
i.e., $V$ is also a nonnegative operator in $L^2(X;\KK^n)$. 

The operator $H+V$ (corresponding to $H'$ in our above considerations) denotes 
the form sum of the two operators. Note that in this generality, the 
corresponding form need not be densely defined. In this case, we obtain a 
selfadjoint operator in the subspace $\hil_0 = \overline{\mathcal D (H^{1/2}) 
\cap \mathcal D (V^{1/2})}$ and extend the corresponding semigroup by $0$ to 
$\hil_0^\perp$. However, in most of the applications we have in mind, such 
complications will not arise since $V$ is bounded.
 
In the following, we use $\|.\|_{\text{HS}}$ to denote the Hilbert-Schmidt norm 
of operators on different spaces. Moreover, for $V$ as above we set
\begin{eqnarray}
\|V\|_{2,\text{HS}}^2 &:=& \int_X \|V(x)\|_{\text{HS}}^2 \drm m(x) \nonumber \\
&=& \sum_{i,j=1}^n \|V_{ij}\|_2^2. \label{hs}
\end{eqnarray}
We write $V \in L^2(X;\KK^{n\times n})$ provided $\|V\|_{2,\text{HS}}$ is 
finite.  Note that in this case $V$ is a bounded operator from $L^\infty(X;\KK^n)$ to $L^2(X;\KK^n)$.\\ 

Before presenting our Hilbert-Schmidt norm estimates on the above semigroup 
difference, we single out one important step from their proof. Here and in the sequel we use the standard shorthand notation
$$
\|T\|_{p,q}:=\| T\|_{\mathcal L (L^p(X;\KK^n),L^q(X;\KK^n))}.
$$

\begin{proposition}\label{prop:2.6} 
 Assume that $T \in \mathcal L (L^2(X;\KK^n),L^\infty(X;\KK^n))$ and $V \in 
L^2(X;\KK^{n \times n})$. Then the operator $VT \in \mathcal L (L^2(X;\KK^n))$ is Hilbert-Schmidt and
$$ \|VT\|_{\HS} \leq \sqrt{n} \cdot \|V\|_{2,\HS} \cdot 
\|T\|_{2,\infty}.$$ 
\end{proposition}

The proof of this proposition relies on the following result from 
\cite{DemuthSSvC-95} (Corollary 2 in that paper), which provides an estimate on 
the trace norm of certain operators on scalar-valued functions: If $A \in 
\mathcal L (L^1(X),L^2(X)), B \in \mathcal L (L^2(X),L^1(X))$ and if there 
exists $\Phi \in L^1(X)$ such that $|Bf| \leq \Phi$ for every $f$ in the unit 
ball of $L^2(X)$, then 
\begin{equation}
\|AB\|_{\text{tr}} \leq \|A\|_{1,2} \|\Phi\|_1.\label{eq:5}
\end{equation}
\begin{proof}[Proof of Proposition \ref{prop:2.6}]  
We consider the scalar case $n=1$ first: Here, using the identity
$$ \|VT\|_{\text{HS}}^2 = \|T^* |V|^2 T \|_{\text{tr}},$$
we can apply the aforementioned Corollary from \cite{DemuthSSvC-95} with 
$A=T^*|_{L^1}, B=|V|^2T$ and $\Phi:=|V|^2 \|T\|_{2,\infty}$. Indeed, since $L^1$ 
is isometrically embedded in $(L^\infty)'$, we have $A \in \mathcal L 
(L^1(X),L^2(X))$ and $\|A\|_{1,2} \leq \|T^*\|_{(L^\infty)',L^2} = 
\|T\|_{2,\infty}$. Moreover, it is easily seen that $B \in \mathcal L 
(L^2(X),L^1(X))$, that $\Phi \in L^1(X)$ with $\|\Phi\|_1 = \|V\|_2^2 
\|T\|_{2,\infty}$ and that  $|Bf| \leq \Phi$ for $f$ in the unit ball of 
$L^2(X)$. Hence, as desired, we obtain from (\ref{eq:5}) that
$$ \|VT\|_{\text{HS}}^2 = \|T^*|V|^2T\|_{\text{tr}} \leq \|T\|_{2,\infty}^2 
\|V\|_2^2.$$  
The vector-valued case is now readily deduced from the scalar one. First, we diagonalize: Since $(V(x))_{x \in X}$ is a measurable family of  hermitian (symmetric) matrices in $\KK^{n\times n}$, we know from \cite[Theorem 2.1]{MR3233068} that there exists a measurable family $(U(x))_{x \in X}$ of unitary (orthogonal) matrices in $\KK^{n \times n}$ such that for every $x \in X$ the matrix
$$ \Lambda(x):= U(x)^* V(x)U(x)$$
is diagonal. 
\begin{remark}
For completeness, we note that the afore mentioned measurable diagonalization is proved in \cite{MR3233068} only for families $(V(x))_{x \in X}$ of positive matrices. However, this easily generalizes to the case considered above. Indeed, first we observe that measurable diagonalization is clearly possible if the family $(V(x))_{x \in X}$ is uniformly bounded below (just consider $V(x) + c \cdot I$ for $c$ sufficiently large). The general case then follows by partitioning $X$ into the union of measurable sets $X_k:=\{ x \in X : k \leq \lambda_1(V(x)) < k+1\}, k \in \ZZ,$ and considering the restrictions of $V$ to these $X_k$'s. That these sets are indeed measurable follows from the measurability of $x \mapsto \lambda_1(V(x))$ (the smallest eigenvalue), which in turn follows from the min-max principle.
\end{remark}
Now for $x \in X$ and $f \in L^2(X;\KK^n)$ we define the unitary operator $U$ on $L^2(X;\KK^n)$ by
$$ (Uf)(x):=U(x)f(x),$$
and we define the operator $\Lambda$ in $L^2(X;\KK^n)$ by
$$(\Lambda f)(x):= \Lambda(x) f(x), \qquad \mathcal D(\Lambda)=\{ f \in L^2(X;\KK^n) : Uf \in \mathcal D(V)\}.$$
Finally, let us define
$$ T_{ij} : L^2(X) \to L^\infty(X), \quad T_{ij} f := ( U^*T(fe_i) | e_j), \qquad 
i,j \in \{ 1, \ldots, n\}, $$
where $e_1,\ldots, e_n$ denote the standard basis vectors of $\KK^n$. For later purposes we note that for $f \in L^2(X)$ we have 
\begin{eqnarray} 
  \|T_{ij}f\|_\infty &=& \esssup_{x \in X} \big| \big(U^*(x)[T(fe_i)](x)|e_j\big)\big| \leq \esssup_{x \in X} |T(fe_i)(x)| \nonumber \\
&\leq& \|T\|_{2,\infty} \|fe_i\|_2 = \|T\|_{2,\infty} \|f\|_2.   \label{eq:10}
\end{eqnarray}
Now we compute 
the Hilbert-Schmidt norm of the operator $VT$ with respect to the orthonormal basis 
$(\varphi_{\alpha} \otimes e_i)_{a \in A, i=1,\ldots,n}$, where 
$(\varphi_\alpha)_{\alpha \in A}$ is some ONB of $L^2(X)$:
\begin{eqnarray*}  
  \|VT\|_{\text{HS}}^2 &=& \|U\Lambda U^*T\|_{\text{HS}}^2 = \|\Lambda U^*T\|_{\text{HS}}^2 \\
&=& \sum_{\alpha \in A} \sum_{i=1}^n \|\Lambda U^*T(\varphi_\alpha 
\otimes e_i)\|_2^2 
= \sum_{\alpha \in A} \sum_{i=1}^n \sum_{j=1}^n \|(\Lambda U^*T(\varphi_\alpha \otimes 
e_i)|e_j)\|_2^2 \\
&=& \sum_{\alpha \in A} \sum_{i=1}^n \sum_{j=1}^n 
\left\|\Lambda_{jj}T_{ij}\varphi_\alpha\right\|_2^2 
= \sum_{i=1}^n \sum_{j=1}^n 
\left\|\Lambda_{jj}T_{ij}\right\|_{\text{HS}}^2.
\end{eqnarray*}
From the scalar case and (\ref{eq:10})  we obtain that
$$ \|\Lambda_{jj}T_{ij}\|_{\text{HS}} \leq \|\Lambda_{jj}\|_2 \|T_{ij}\|_{2,\infty} \leq 
\|\Lambda_{jj}\|_2 \|T\|_{2,\infty}.$$
Hence, the previous identities imply that
\begin{eqnarray*}
  \|VT\|_{\text{HS}}^2 &\leq& n \|T\|_{2,\infty}^2 \sum_{j=1}^n 
 \|\Lambda_{jj}\|_{2}^2 = n \|T\|_{2,\infty}^2 \|\Lambda\|_{2,\text{HS}}^2=n \|T\|_{2,\infty}^2 \|V\|_{2,\text{HS}}^2.
\end{eqnarray*}
This concludes the proof of the proposition. 
\end{proof}
In the following we present several estimates on the Hilbert-Schmidt norm of the 
difference of the semigroups $e^{-tH}$ and $e^{-t(H+V)}$, which are defined as 
above. We begin with two results for the case of bounded $V$, which will be used 
later on in our estimate on the first Betti number of compact manifolds. After 
that, we extend one of these results to unbounded potentials. While unbounded 
potentials will not play a role in the present article, we include them for 
completeness and because they might well become important in future work (for 
instance, when considering the case of non-compact manifolds).
\begin{remark}  
  Hilbert-Schmidt norm estimates for semigroup differences of operators in 
$L^2(X)$ have been studied in a variety of contexts and the literature on the 
subject is extensive. For an overview and many references we refer to the 
monograph \cite{MR1772266}, which studies such estimates for generators of 
Feller semigroups. However, as far as we can say, the following results are the 
first estimates concerning operators on vector-valued functions. Moreover, in 
the stated generality, we think that they  might even be new in the scalar 
case. 
\end{remark}
 The natural setting of our first result is when the involved 
semigroups $(T(t))_{t\ge 0}$ are \emph{ultracontractive},
i.e., for every $t>0$ they map $L^2(X; 
\KK^n)$ to $L^\infty(X;\KK^n)$ and  $\|T(t)\|_{2,\infty}<\infty$.  However, we need this additional property for one $t_0$ only.

\begin{theorem}\label{thm:truncated} 
Assume that $H \geq 0$ is a selfadjoint operator in $L^2(X;\KK^{n})$ and $V 
\in L^2\cap L^\infty(X;\KK^{n \times n})$. Moreover, suppose 
that for some $t_0>0$ we have $e^{-t_0H},e^{-t_0(H+V)}\in \mathcal L (L^2(X;\KK^n),L^\infty(X;\KK^n))$. Then
{\small
\begin{eqnarray}
 \|e^{-2t_0H}-e^{-2t_0(H+V)}\|_{\HS}  &\leq& \sqrt{n} \; \|V\|_{2,\HS} 
\left(\|e^{-t_0H}\|_{2,\infty} + \|e^{-t_0(H+V)}\|_{2,\infty} \right) \cdot 
\nonumber\\
&& \cdot \int_0^{t_0} \|e^{-s(H+V)}\|_{2,2} \:\drm s.  \label{eq:6} 
\end{eqnarray}
}

\end{theorem}

 \begin{proof} 
By the Duhamel principle, see \cite{MR710486}, formula (1.8) on page 78, we 
obtain that
\begin{eqnarray*}
  && e^{-2t_0 H}- e^{-2t_0(H+V)} = \int_0^{2t_0} e^{-(2t_0-s)(H+V)}Ve^{-sH}\drm s \\
&=& \int_0^{t_0} e^{-(2t_0-s)(H+V)}Ve^{-sH}\drm s + \int_{t_0}^{2t_0} 
e^{-(2t_0-s)(H+V)}Ve^{-sH}\drm s
\end{eqnarray*}
and hence 
\begin{eqnarray*}
  \| e^{-2t_0 H}- e^{-2t_0(H+V)}\|_{\text{HS}} &\leq& 
 \int_0^{t_0} \| e^{-(2t_0-s)(H+V)}Ve^{-sH} \|_{\text{HS}} \:\drm s \\
&& + \int_{t_0}^{2t_0} \| e^{-(2t_0-s)(H+V)}Ve^{-sH} \|_{\text{HS}} \:\drm s.
\end{eqnarray*}
We estimate the integrals separately but with the same idea and start with the 
second one: With a change of variables and an application of Proposition 
\ref{prop:2.6} (with $T=e^{-t_0H}$) we obtain
\begin{eqnarray*}
&& \int_{t_0}^{2t_0} \| e^{-(2t_0-s)(H+V)}Ve^{-sH} \|_{\text{HS}} \:\drm s = 
\int_{0}^{t_0} \| e^{-(t_0-r)(H+V)}Ve^{-(t_0+r)H} \|_{\text{HS}} \:\drm r \\
&\leq& \int_{0}^{t_0} \| e^{-(t_0-r)(H+V)}\|_{2,2} \|Ve^{-t_0H} \|_{\text{HS}} 
\|e^{-rH}\|_{2,2} \:\drm r \\
&\leq& \sqrt{n}\|V\|_{2,\text{HS}} \|e^{-t_0H}\|_{2,\infty} \int_{0}^{t_0} \| 
e^{-(t_0-r)(H+V)}\|_{2,2} \|e^{-rH}\|_{2,2} \: \drm r \\
&\leq& \sqrt{n}\|V\|_{2,\text{HS}} \|e^{-t_0H}\|_{2,\infty} \int_0^{t_0} \| 
e^{-s(H+V)}\|_{2,2} \drm s.
\end{eqnarray*}
In the last inequality we used a change of variables and the fact that for all 
$t \geq 0$ we have $\|e^{-tH}\|_{2,2} \leq 1$. Finally, for the first integral 
we use that $\|A\|_{\text{HS}} = \|A^*\|_{\text{HS}}$ and obtain in a similar 
fashion that
  \begin{eqnarray*}
&& \int_0^{t_0} \| e^{-(2t_0-s)(H+V)}Ve^{-sH} \|_{\text{HS}} \:\drm s  = 
\int_0^{t_0} \|e^{-sH}V e^{-(2t_0-s)(H+V)} \|_{\text{HS}} \:\drm s \\
&\leq& \sqrt{n}\|V\|_{2,\text{HS}} \|e^{-t_0(H+V)}\|_{2,\infty} \int_{0}^{t_0} 
\|e^{-sH}\|_{2,2} \| e^{-(t_0-s)(H+V)}\|_{2,2}  \: \drm s \\
&\leq& \sqrt{n}\|V\|_{2,\text{HS}} \|e^{-t_0(H+V)}\|_{2,\infty} \int_0^{t_0} \| 
e^{-s(H+V)}\|_{2,2} \:\drm s.
  \end{eqnarray*}
This concludes the proof.
 \end{proof} 
A slightly unpleasant feature of the previous theorem is the fact that we need 
to assume that also the perturbed semigroup is ultracontractive, which might 
be difficult to check. In the scalar case, however, this does usually not pose a problem, since in this case the boundedness of $e^{-t_0(H+V)}: L^2(X) \to 
L^\infty(X)$ automatically follows from the corresponding boundedness of 
$e^{-t_0H}$ provided the last semigroup is 
positivity preserving. Also, we just note that here the additional assumption that $V$ is bounded is not necessary.

To overcome the described problem in the vector-valued case, we now make the additional assumption that 
there is a dominating semigroup on the scalar $L^2$--space. This requires some terminology:

Let $H_0$ be a selfadjoint lower-semibounded operator in $L^2(X)$. We say that its semigroup 
$(e^{-t H_0})_{t\ge 0}$ \textbf{dominates} $(e^{-tH})_{t\ge 0}$ if the following relation is satisfied for all $t>0$:
$$\vert e^{-tH}f\vert(x)\leq e^{-tH_0}\vert f\vert(x), \quad (x\in X, f\in 
L^2(X;\KK^n)).$$  
This implies $(e^{-tH_0})_{t\ge 0}$ is positivity preserving (i.e., for all $t\ge 0$: $e^{-tH_0}f \geq 0$ if $f\geq 0$). 
\begin{remark} 
  Domination of semigroups has a long history and its beginnings are
  intimately related to the situation we have in mind, the semigroups
  of the Laplace--Beltrami operator on functions and the semigroup of
  the Hodge--Laplacian on $1$--forms. See
  \cite{HessSchraderUhlenbrock-77,Simon-77} for early results. In
  \cite{HessSchraderUhlenbrock-80} the form characterization of
  domination was used to show that the semigroup of the
  Bochner-Laplacian on a Riemannian manifold is in fact dominated by
  the semigroup of the Laplace-Beltrami operator. Furthermore, it was shown 
  that under suitable conditions, the Hodge-deRham Laplacian
  is dominated by a Schr\"odinger operator generated by the
  Laplace-Beltrami plus a suitable potential depending on Ricci
  curvature, as we will use later. Especially in Riemannian geometry,
  this fact has been used extensively to study geometric and
  topological properties of manifolds as well as properties of the semigroup and corresponding heat
  kernel of generalized Schr\"odinger operators on vector
  bundles, see \cite{Gueneysu-14,Gueneysu-10,
    Gueneysu-16,ElworthyRosenberg-91,RoseStollmann-15,Rose-17,Rose-16,Rose-16b,CoulhonZhang-07}
  and the references therein. For a recent survey, see Section 2 in 
  \cite{RoseStollmann-18}; for an abstract point of view and a more
  thorough discussion of the literature on  semigroup domination, see the recent \cite{LenzSchmidtWirth-17}.
\end{remark}

\begin{theorem}\label{thm:ultra} 
Assume that $H\geq 0$ is a selfadjoint operator in $L^2(X;\KK^n)$ and that $V\in 
L^2(X;\KK^{n\times n})$ with $V \succeq 0$.  Moreover, let  $H_0$ be selfadjoint and lower-semibounded in $L^2(X)$ such that $(e^{-tH_0})_{t\ge 0}$ dominates $(e^{-tH})_{t\ge 0}$ and in addition
$e^{-t_0H_0}\in \mathcal L (L^2(X),L^\infty(X))$ for some $t_0>0$. Then
 \begin{equation}
  \label{eq:80}
  \|e^{-2t_0H}-e^{-2t_0(H+V)}\|_{\HS} \leq 2\, \sqrt{n} \|V\|_{2,\HS} 
\|e^{-t_0H_0}\|_{2,\infty} \cdot t_0.
\end{equation}
If, furthermore, $V$ is bounded, also the following estimate holds true:
 \begin{equation}
  \label{eq:8}
  \|e^{-2t_0H}-e^{-2t_0(H+V)}\|_{\HS} \leq 2\, \sqrt{n} \|V\|_{2,\HS} 
\|e^{-t_0H_0}\|_{2,\infty} \int_0^{t_0} \|e^{-s(H+V)}\|_{2,2}\: \drm s.
\end{equation}
\end{theorem}

\begin{remark}
Given the above assumptions on $H$ and $V$, the operator $H+V$ is nonnegative. 
In our later applications, this operator will even be positive. Taking this into 
account, it is good to observe that if $H+V \geq \rho_0$ for some $\rho_0 \geq 
0$, then the spectral theorem implies
\begin{equation}
\int_0^{t_0} \|e^{-s(H+V)}\|_{2,2} \:\drm s \leq \int_0^{t_0} e^{-\rho_0 s} \drm s = 
\left\{
  \begin{array}{cl}
    t_0, & \rho_0 = 0, \\[4pt]
    \frac 1 {\rho_0} \left( 1 - e^{-t_0 \rho_0} \right), & \rho_0 > 0.
  \end{array}\right. 
\label{eq:3}
\end{equation}

\end{remark}

\begin{proof}[Proof of Theorem \ref{thm:ultra}]
We first treat the case that $V$ is bounded and 
use the fact that $V \succeq 0$ (i.e., $V(x) \geq 0$ for all $x \in X$), which implies that for $t>0$ arbitrary the operator norm of the matrix $e^{-tV(x)} : \KK^n \to \KK^n$ is at 
most $1$. This implies the following estimate for $f \in L^2(X;\KK^n)$ and $x \in X$:
$$ |e^{-tV}f|(x) = |e^{-tV(x)}f(x)| \leq \|e^{-tV(x)}\| |f(x)| \leq |f(x)|.$$
Moreover, since $e^{-tH_0}$ is positivity 
preserving, we can use the previous estimate and the domination property to obtain 
$$ |e^{-tH}e^{-tV}f|(x) \leq e^{-tH_0}|e^{-tV}f|(x) \leq e^{-tH_0} |f| 
(x).$$
By induction, we thus see that for all $k\in\NN$ we have the pointwise 
inequality 
$$\vert \left(e^{-\frac{t}{k}H}e^{-\frac{t}{k}V}\right)^k 
f\vert\leq\left(e^{-\frac{t}{k}H_0}\right)^k\vert f\vert = e^{-{t}H_0}\vert f\vert,\quad f\in 
L^2(X;\KK^n).$$
The Trotter product formula and the assumption that $e^{-t_0H_0} : L^2 \to L^\infty$ is bounded imply that
\begin{equation}
\Vert e^{-t_0(H+V)}\Vert_{2,\infty}\leq \Vert 
e^{-t_0H_0}\Vert_{2,\infty}.\label{eq:4}
\end{equation}
In particular, we also obtain 
\begin{equation}
 \Vert e^{-t_0H}\Vert_{2,\infty}\leq \Vert 
e^{-t_0H_0}\Vert_{2,\infty}.\label{eq:7}
\end{equation}
Hence, the assumptions of Theorem \ref{thm:truncated} are satisfied and 
(\ref{eq:8}) immediately follows from (\ref{eq:6}), using the above estimates 
(\ref{eq:4}) and (\ref{eq:7}).
Thus, we have established estimate (\ref{eq:8})  for bounded potentials.

For general $V \in L^2(X;\KK^n)$ 
let us introduce
 $$ V^{(k)}(x):= 1_{\{ x \in X : \|V(x)\| \leq k\}}(x) \cdot V(x), \quad k \in \NN.$$
Then $V \succeq V^{(k)} \succeq 0$ and $V^{(k)} \to V$ pointwise for $k \to 
\infty$. Moreover, 
 $$ H + V^{(k)} \overset{\text{srs}}{\to} H + V \quad \text{for } k \to 
\infty,$$ 
 by monotone form convergence, see, e.g., \cite{MR0500266}; here 'srs' refers to 
convergence in the strong resolvent sense, so that we have strong convergence
of semigroups as well:
 \begin{equation}  
  e^{-t(H + V^{(k)})} \overset{\text{s}}{\to} e^{-t(H + V)} \quad \text{for } k 
\to \infty.    \label{eq:12}
 \end{equation}
Since each $V^{(k)}$ is bounded and $\|V^{(k)}\|_{2,\text{HS}} \leq 
\|V\|_{2,\text{HS}}$, we know from the first part of the proof, also using (\ref{eq:3}), that for all $k \in \NN$
$$   \|e^{-2t_0H}-e^{-2t_0(H+V^{(k)})}\|_{\text{HS}} \leq 2\, \sqrt{n} 
\|V\|_{2,\text{HS}} \|e^{-t_0H_0}\|_{2,\infty} \cdot t_0.$$
Noting that this bound is uniform in $k$, we can use the fact that the 
Hilbert-Schmidt operators $(\mathcal S_2, \|.\|_{\text{HS}})$ are a Hilbert 
space, and hence the corresponding unit ball is weakly compact, to see that 
there exists a subsequence $(k_l)_{l \in \NN}$ and $T \in \mathcal S_2$ such 
that 
\begin{equation}
    \label{eq:11}
  e^{-2t_0 H}- e^{-2t_0(H+V^{(k_l)})} \to T \quad \text{for } l \to \infty
\end{equation}
weakly, and 
$$ \|T\|_{\text{HS}} \leq \liminf_l  \|e^{-2t_0 H}- 
e^{-2t_0(H+V^{(k_l)})}\|_{\text{HS}} \leq 2\, \sqrt{n} \|V\|_{2,\text{HS}} 
\|e^{-t_0H_0}\|_{2,\infty} \cdot t_0.$$
(\ref{eq:12}) shows that $T=e^{-2t_0 H}- e^{-2t_0(H+V)}$ and hence the estimate (\ref{eq:80}) follows. 
\end{proof}

\section{Bounds on the first Betti number} 
\label{sec3}
Let us now introduce the set--up to which we apply the results of the preceding 
section. We fix a compact $n$--dimensional Riemannian manifold $(M,g)$ and use the
Riemannian volume element to define the corresponding $L^2$--spaces.

Next, we introduce the Laplace--Beltrami operator

$$
\Delta=\delta d \ge 0,
$$
which, according to our sign convention, is a non-negative selfadjoint operator in $L^2(M)$.
Moreover, we consider the Hodge--Laplacian
$$
\Delta^1=\delta d + d\delta \ge 0
$$ 
acting on $1$-- forms, so that it is a non-negative selfadjoint operator in the Hilbert space $L^2(M;\Omega^1)$
of square integrable sections of the cotangent bundle. We will frequently identify $L^2(M;\Omega^1)=L^2(M;\RR^n)$.

The Weitzenböck formula gives that
$$
\Delta^1=\nabla^*\nabla+\Ric ,
$$ 
where the latter term gives a matrix--valued potential $M\ni x\mapsto \Ric_x$, and $\Ric_x$ is the
Ricci tensor interpreted as an endomorphism of the cotangent space $\Omega^1_x(M):=(T_xM)^*$. The former term, $\nabla^*\nabla$, is the
rough or Bochner--Laplacian. Note that $\Ric_x$ is given by a symmetric matrix with entries varying smoothly in $x$ so that
$$
\Ric\in L^\infty(M;\RR^{n\times n})\subset L^2(M;\RR^{n\times n}) 
$$
and we can substitute $\Ric$ for $V$ of the preceding section. It follows from the Hodge theorem, see \cite{Jost-08}, Thm. 2.2.1, that we can identify
the first real cohomology group with the space of harmonic $1$--forms,
$$
H^1(M)\simeq \Ker(\Delta^1) ,
$$
so that the first Betti number equals
$$
b_1(M)=\dim\left(\Ker(\Delta^1)\right) .
$$
In view of what we studied in the previous section, it is therefore
natural to consider $H:=\Delta^1$. The appropriate comparison operator $H'$ is constructed next. We fix $\rho_0>0$ and consider
the symmetric matrix $\Ric_x-\rho_0$, where we omit the identity matrix in the last expression. We can decompose
$$
\Ric_x-\rho_0=(\Ric_x-\rho_0)_+-(\Ric_x-\rho_0)_-,
$$
where
$(\Ric_x-\rho_0)_+$ is non--negative definit, $(\Ric_x-\rho_0)_-$ is positive definite and their product is $0$. This decomposition
is measurable in $x$, since the involved matrices are continuous as functions of $x$.

We put

\begin{equation}
H':= \Delta^1+(\Ric_x-\rho_0)_- .\label{eq:13}
\end{equation}
Note that 
$$
V:=(\Ric-\rho_0)_- \succeq 0
$$
in the sense of the previous section.
Moreover
\begin{eqnarray*}
H'&=& \Delta^1+(\Ric-\rho_0)_- \\
&=& \nabla^*\nabla+\Ric -\rho_0+\rho_0+(\Ric-\rho_0)_- \\
&=&\nabla^*\nabla+(\Ric-\rho_0)_++\rho_0\\
&\ge& \rho_0 .
\end{eqnarray*}
As a first application of our abstract results we note that Proposition \ref{prop:birman} readily implies:
\begin{corollary}
 For $\rho_0>0$, $t_0>0$ and $p \geq 1$:
 \begin{equation}\label{bettischatten}
 b_1(M)\le \left(1-e^{-2\rho_0 t_0}\right)^{-p}\left\| e^{-2t_0\Delta^1}-e^{-2t_0(\Delta^1+(\Ric-\rho_0)_-)}
 \right\|_{\mathcal{S}_p}^p .
   \end{equation}
\end{corollary}
Note that the semigroups involved consist of trace class operators, so that the RHS of
\eqref{bettischatten} is finite for all values of the parameters involved.

We go on to specialize to $p=2$, using Theorem \ref{thm:ultra}. To this end we recall the following results 
from \cite{HessSchraderUhlenbrock-80}
that we mentioned already, where the authors use the opposite sign convention!
We define
\begin{equation}
\rho(x):=\min\sigma(\Ric_x), \quad x \in M.\label{eq:14}
\end{equation}
\begin{proposition}
 For the operators defined above, we have the following domination
 of the corresponding semigroups:
 \begin{itemize}
  \item[\emph{(1)}] For all $\omega\in L^2(M;\Omega^1)$ and $t\ge 0$:
  $$
  \left| e^{-t\nabla^*\nabla}\omega\right|\le e^{-t\Delta}|\omega| .
  $$
  \item[\emph{(2)}] For all $\omega\in L^2(M;\Omega^1)$ and $t\ge 0$:
  $$
  \left| e^{-t\Delta^1}\omega\right|\le e^{-t(\Delta+\rho)}|\omega| .
  $$
 \end{itemize}
\end{proposition}

Therefore we can apply the $p=2$ case of Theorem \ref{thm:ultra}, in particular (\ref{eq:8}) and (\ref{eq:3}),  with the above bound \eqref{bettischatten} 
to obtain the following estimate, putting $H_0:=\Delta+\rho$:
\begin{theorem}
 \label{main}
 For $\rho_0>0$ and $t_0>0$: 
 \begin{equation}\label{bettimain}
 b_1(M)\le \frac{4n}{(\rho_0(1+e^{-t_0\rho_0}))^2} \left\| (\Ric_\cdot-\rho_0)_-\right\|_{2,\HS}^2\left\| e^{-t_0(\Delta+\rho)}\right\|^2_{2,\infty} . 
   \end{equation}
\end{theorem}

We want to point out again that earlier results on Betti number bounds with negative curvature assumptions always depended implicitly on an a priori bound at least on the heat kernel or even a Sobolev constant. In contrast, our result shows that such a control is indeed not needed if we have bounds on certain norms of the perturbed heat semigroup for one $t_0>0$, which could be achieved using other techniques. 

Let us provide at least one example showing how the above bound (\ref{bettimain}) can be made more explicit if further information on the heat kernel is available. Here we use an upper bound on the kernel  going back to the celebrated
paper \cite{LiYau-86}.

\begin{corollary}\label{cor:main}
 Assume that $\Ric\ge -K$, where $K\ge 0$ and denote by $D$ the diameter of $M$. Then
 \begin{equation}
 b_1(M)\le c(n) \rho_0^{-2} \left\| (\Ric_\cdot-\rho_0)_-\right\|_{2,\HS}^2\Vol(M)^{-1}e^{\alpha(n)KD^2},
 \end{equation}
where $c(n),\alpha(n)$ depend on $n$ only.
\end{corollary}

\begin{proof}
 First note that
 $$
 \left\| e^{-t_0(\Delta+\rho)}\right\|^2_{2,\infty}\le e^{2t_0K}\left\| e^{-t_0\Delta}\right\|^2_{2,\infty}
 $$
 since $\rho\ge -K$. Denoting the heat kernel by $p$, we get
 $$
 \left\| e^{-t_0\Delta}\right\|^2_{2,\infty}=\esssup_{x\in M}\| p(t_0;x,\cdot)\|^2_2 .
 $$
We now infer the heat kernel estimate from Corollary 3.1 in \cite{LiYau-86}, setting $t_0=D^2$ so that the 
balls appearing equal $M$, whence

$$
|p(t_0;x,\cdot)|^2\le c(n)\Vol(M)^{-2}e^{\alpha(n)KD^2} .
$$
Integrating this pointwise bound we arrive at

$$
\|p(t_0;x,\cdot)\|_2^2\le c(n)\Vol(M)^{-1}e^{\alpha(n)KD^2} .
$$
Combined with the above observation and adapting the dimension dependent constants, estimate (\ref{bettimain}) proves the claim.
\end{proof}

Finally, let us mention that other results on heat kernel bounds can easily be applied in a similar manner to Theorem \ref{main}, instead of the very crude 
one we employed in the preceding corollary. We conclude the paper with the following list of examples, stating conditions on $M$ where quantitative heat kernel bounds can be obtained. For more details, the reader should consult \cite{Grigoryan-09, RoseStollmann-18, Rose-17, Rose-16, Carron-16}.

\begin{example} The following assumptions admit quantitative bounds on the heat kernel:
\begin{enumerate}[(i)]

\item $M$ is geodesically complete, of dimension at least two, and 
possesses a distance function $r_\xi$, $\xi\in M$, such that 
$$\forall x,\xi \in M\colon \quad \vert \nabla r_\xi(x)\vert\leq 1, \quad \Delta 
r_\xi(x)\geq 2n,$$

\item $M$ is an $n$-dimensional minimal submanifold of $\RR^N$, $N>n$, i.e., all its mean curvature vectors vanish 
identically on $M$,
 
\item $M$ is a Cartan-Hadamard manifold, i.e., it is simply connected and its 
sectional curvature is non-positive everywhere,

\item $M$ is an $n$-dimensional manifold of bounded geometry, i.e., its Ricci 
curvature is bounded below and its injectivity radius is positive,
\item $M$ is complete and of dimension $n\geq 2$, the injectivity radius is positive, and there exist an $r\in (0,\diam(M))$ and a $p>n/2$ such that the quantity 
$$\sup_{x\in M} \frac{1}{\Vol(B(x,r))}\int_{B(x,r)}\rho_-^p\, \dvol$$
is small enough,
\item 
$M$ is of dimension $n\geq 3$, has bounded diameter and Ricci curvature bounded below, 
\item $M$ is of dimension $n\geq 2$, has bounded diameter, and we have
$$\int_0^{\diam M^2}\Vert \euler^{-t\Delta}\rho_-\Vert_\infty\drm t\leq \frac 1{16n}.$$
\end{enumerate}
\end{example}

\def\cprime{$'$}

\end{document}